\documentclass[a4paper,11pt,english]{smfart}
\usepackage{amsfonts}
\usepackage{amsmath,amsthm}
\usepackage{latexsym}
\usepackage{array}
\usepackage{amssymb}
\usepackage{graphicx}
\usepackage{enumerate}
\usepackage{verbatim}
\usepackage{float}
\usepackage[english]{babel}

\usepackage{color}

\newtheorem{thm}{Theorem}{\bf}{\it}
\newtheorem{lem}[thm]{Lemma}{\bf}{\it}
{\bf}{\it}
{\bf}{\it}
{\bf}{\it}
{\bf}{\it}
{\bf}{\it}
{\bf}{\it}
\newtheorem{cor}[thm]{Corollary}{\bf}{\it}
{\bf}{\it}
{\bf}{\it}

\theoremstyle{definition}
\newtheorem{defn}[thm]{Definition}{\bf}{\rm}
\newtheorem{ex}[thm]{Example}{\bf}{\rm}
{\bf}{\rm}
\newtheorem{rem}[thm]{Remark}{\bf}{\rm}

\newcommand{\D}{\mathcal{D}}

\def\sfrac#1#2{\mbox{\footnotesize$\displaystyle\frac{#1}{#2}$}}

\def\sfrac#1#2{\mbox{\footnotesize$\displaystyle\frac{#1}{#2}$}}

\newcommand{\C}{\mathcal{C}}

\makeatletter
\def\@makefnmark{\hbox{$\m@th^{\@thefnmark}$}}
\makeatother

\author{Erwan Faou}
\address{INRIA \& ENS Cachan Bretagne, Avenue Robert Schumann, F-35170 Bruz, France. }
\email{erwan.faou@inria.fr}

\author{Alexander Ostermann}
\address{Institut f\"ur Mathematik, Universit\"at Innsbruck, Technikerstr.~13, A-6020 Innsbruck, Austria. }
\email{alexander.ostermann@uibk.ac.at}

\author{Katharina Schratz}
\address{INRIA \& ENS Cachan Bretagne, Avenue Robert Schumann, F-35170 Bruz, France. }
\email{katharina.schratz@inria.fr}

\title[Exponential splitting methods]{Analysis of exponential splitting methods\\ for  inhomogeneous  parabolic equations}

\begin{document}

\begin{abstract}
We analyze the convergence of the exponential Lie and exponential Strang splitting applied to  inhomogeneous  second-order parabolic equations  with Dirichlet boundary conditions.   A recent result on the smoothing properties of these methods allows us to prove sharp convergence results in the case of homogeneous Dirichlet boundary conditions.  When no source term is present  and natural regularity assumptions are imposed   on the initial value, we show full-order convergence of both methods. For inhomogeneous equations, we prove full-order convergence for the exponential Lie splitting, whereas order reduction to $1.25-\varepsilon$ for arbitrary small $\varepsilon>0$ for the exponential Strang splitting. Furthermore, we give sufficient conditions on the inhomogeneity for full-order convergence of both methods. Moreover our theoretical convergence results explain the severe order reduction to $0.25 -\varepsilon$ for every $\varepsilon>0$ of splitting methods applied to problems involving inhomogeneous Dirichlet boundary conditions. We include numerical experiments to underline the sharpness of our theoretical convergence results.
\end{abstract}

\thanks{}

\maketitle
\section{Introduction}
The focus of this work lies in a rigorous error analysis of the exponential Lie and Strang splitting applied to second-order inhomogeneous parabolic evolution equations
\begin{equation}
\begin{aligned}\label{eq:model}
& \partial_t w(t,x,y) =\mathcal{L}(\partial_x,\partial_y)w(t,x,y)+\psi(t,x,y),\quad (x,y) \in \Omega = (0,1)^2, \ t\in ]0,T],\\
& w(0,x,y) = w_0(x,y),\\
& w(t,\cdot,\cdot)\big\vert_{\partial\Omega} = 0 \quad \text{for all }t\in [0,T],
\end{aligned}
\end{equation}
where $\mathcal{L}(\partial_x,\partial_y)=\partial_x \big(a(x,y)\partial_x\big)+\partial_y\big(b(x,y)\partial_y\big)$ denotes a strongly elliptic differential operator with sufficiently smooth (positive) coefficients $a$ and $b$. In particular, we elucidate the following phenomena: for arbitrary inhomogeneities $\psi$ the exponential Strang splitting  naturally induced by the decomposition of $\mathcal{L}$   suffers from an order reduction whereas the exponential Lie splitting is full-order convergent, see Figure~\ref{figure1}  in Section 5   for a numerical example. Nevertheless, for a certain choice of $\psi$ also the Strang splitting reaches full second-order convergence, see Figure~\ref{figure2}. A recent result on the smoothing properties of exponential splitting methods, see \cite{OstS12}, allows us to prove this order behavior rigorously. We can show full-order convergence of the exponential Lie splitting method,  despite the fact that  its local order of consistency  reduces to $0.25-\varepsilon$ for arbitrary small $\varepsilon>0$. Furthermore we can state a sharp result on the global order reduction to $1.25 - \varepsilon$ for arbitrary small $\varepsilon>0$ of the exponential Strang splitting applied to equations with general inhomogeneities $\psi$.

Operator splitting methods have gained a lot of attention in recent years, as in many situations they allow a keen reduction of the computational cost. For a general introduction to splitting methods, we refer to \cite{HunV2003}, \cite{McLacQ02} and references therein. Particularly, exponential splitting methods with relatively bounded operators were studied in \cite{JahL00}; the situation of two unbounded operators in the homogeneous parabolic case was studied in \cite{HanO09} and \cite{HanO09H}. However, due to the recent smoothing result, we can improve the error bounds in the homogeneous parabolic setting given in the last two references. Resolvent splitting applied to inhomogeneous equations was discussed in \cite{OstS12b}.

This work in particular distinguishes from the above as inhomogeneous Dirichlet boundary conditions are included in the analysis. It is well know that inhomogeneous Dirichlet boundary conditions generally lead to a severe order reduction of splitting methods,\footnote[1]{\emph{Particular attention will be paid to the order reduction caused by boundary conditions since that is often the main reason for a disappointing convergence behavior with splitting methods}; \cite[p.\;348]{HunV2003}.} where  for an introductory reading we refer to \cite[Section IV.2]{HunV2003}. Our theoretical analysis allows us a precise explanation of these phenomena, and numerical experiments confirm the sharpness of the derived convergence bounds.

For the analysis, the model problem \eqref{eq:model} is formulated as an abstract evolution equation in $L^2(\Omega)$, where for convenience we denote by $\Vert\cdot\Vert$ the $L^2(\Omega)$ as well as the $L^2(\Omega)$-operator norm. Thus, the model problem reads
\begin{equation}\label{eq:abM}
\begin{aligned}
&u'(t) = Lu(t) + g(t) = Au(t)+Bu(t)+g(t), \quad t\in ]0,T],\\
&u(0) = u_0
\end{aligned}
\end{equation}
with $u(t) = w(t,\cdot,\cdot)$, $Au = \partial_x \big(a\partial_x u\big)$, $Bu = \partial_y \big(b\partial_y u\big)$, and $g(t) = \psi(t,\cdot,\cdot)$.

For the numerical time integration, we apply the so called dimension splitting, where the differential operator $L$ is split along its dimensions. Let $h>0$ denote the step size. At time $t_{n+1} = (n+1)h$ the exact solution of \eqref{eq:abM} is approximated by exponential Lie splitting
\begin{equation}\label{Lie}
u_{n+1} = \mathrm{e}^{hA}\mathrm{e}^{hB}\big(u_n+hg(t_n)\big)
\end{equation}
or the exponential Strang splitting
\begin{equation}\label{Strang}
u_{n+1} = \mathrm{e}^{\frac{h}{2}A}\mathrm{e}^{\frac{h}{2}B}\Big(\mathrm{e}^{\frac{h}{2}B}\mathrm{e}^{\frac{h}{2}A} u_n+hg(t_n+\textstyle\frac{h}{2})\Big),
\end{equation}
respectively. In Section \ref{sec:conv} the advantage of choosing this particular structure of the exponential Strang splitting is discussed in more detail.

In the bounded setting one can easily see via Taylor series expansion that the \emph{classical order} of the exponential Lie splitting \eqref{Lie} and the exponential Strang splitting \eqref{Strang} is one and two, respectively. Under particular assumptions on the initial value $u_0$ and inhomogeneity $g$ we will show that full-order convergence of both methods also holds true in our unbounded setting, see Section \ref{sec:conv}, Theorems \ref{thm:inhLie} and \ref{thm:inhStrang}, respectively. In particular we can explain the numerical observed full-order convergence of the exponential Strang splitting in \cite{HanO09H}. In the situation of an arbitrary inhomogeneity $g$, we will prove full-order convergence for the exponential Lie splitting, whereas order reduction to $1.25-\varepsilon$ for arbitrary small $\varepsilon>0$ for the exponential Strang splitting, see Section \ref{sec:conv}, Theorems \ref{thm:inhLie} and \ref{thm:inhStrang}, respectively. In Section \ref{inhB} we explain the disappointing performance of splitting methods applied to equations involving inhomogeneous Dirichlet boundary conditions by applying our newly derived theoretical results, see Corollaries~\ref{cor:bc0} and \ref{cor:bc}. Finally, in Section \ref{num} we illustrate the sharpness of our theoretical convergence results with numerical experiments.

We commence with a section on the analytic framework.

\section{Analytic framework}
In order to start with the convergence analysis we need some functional analytic results on the domains of the operators in \eqref{eq:abM}. In the following we set
\begin{equation*}
\begin{aligned}
& \D(L) = H^2(\Omega)\cap H_0^1(\Omega),\\
& \D(A) = \{u\in L^2(\Omega)\,;\; \partial_{xx}u,\partial_x u \in L^2(\Omega),\ u(0,y) = u(1,y) = 0 \text{ f.a.e } y \in (0,1)\},\\
& \D(B) = \{u\in L^2(\Omega)\,;\; \partial_{yy}u,\partial_y u \in L^2(\Omega),\ u(x,0) = u(x,1) = 0 \text{ f.a.e } x \in (0,1)\}.
\end{aligned}
\end{equation*}
We will also need the compositions of the operator $L$ and therefore introduce the spaces
$$
\D(L^k) =\big\{ u\in \D(L^{k-1})\,;\, Lu\in\D(L^{k-1})\big\}, \quad k=1,2,3,\ldots
$$
Note that all the appearing operators $(\D(A),A)$, $(\D(B),B)$ and $(\D(L),L)$ generate analytic semigroups of contractions on $L^2(\Omega)$, see \cite{OstS12}. For a general introduction to the theory of analytic semigroups it is referred to \cite{Pazy1983}. Furthermore note that $\D(L^\gamma)$ is free from boundary conditions if and only if $\gamma < \frac{1}{4}$, see \cite{Fujiwara67}.

In the following we state some results on the regularity and compatibility between the domains of the full operator $L$, the split operators $A$, $B$, and their compositions, which we will need later on in our convergence analysis.
\begin{lem}\label{lem:L3}
The following regularity results hold
\begin{equation*}
\D(L^2)\subset H^4(\Omega)\quad\text{and}\quad \D(L^3)\subset H^{5-\varepsilon}(\Omega)\quad \text{for all } \varepsilon>0.
\end{equation*}
\end{lem}
\begin{proof}
 This result can be derived from an analysis of the classical compatibility conditions for a strongly elliptic operator with Dirichlet boundary conditions on a square domain, see for instance \cite{Grisvard1985}.  A complete proof can be found in \cite{HeSW}.
\end{proof}
As the operator $(\D(L),L)$ is considered on the unit square $\Omega$, i.e., a corner domain, the standard regularity results do not hold true. This is due to the so called corner singularities that arise when solving partial differential equations on non-smooth domains. In particular, for general coefficients, $\D(L^3)\nsubseteq H^6(\Omega)$.
\begin{lem}\label{lem:ABL2}
The domains of the operators satisfy the following compatibility condition
\begin{equation*}
\D(L^2)\subseteq \D(A^2)\cap\D(AB)\cap\D(BA)\cap\D(B^2).
\end{equation*}
\end{lem}
\begin{proof}
 From  Lemma \ref{lem:L3} we have
\begin{equation*}
\D(L^2) \subseteq\left\{ u\in H^4(\Omega)\cap H_0^1(\Omega)\,;\, Lu\vert_{\partial\Omega} = 0\right\}.
\end{equation*}
Let $u\in \D(L^2)$. Then $Au,Bu \in H^2(\Omega)$.  As $\Omega \subset \mathbb{R}^2$, we have $H^4(\Omega)\cap  H_0^1(\Omega) \subset \C^2(\overline{\Omega})\cap H_0^1(\Omega)$ using standard Sobolev embeddings results. Hence as $u\in H^4(\Omega)\cap H_0^1(\Omega)$  we have for $x_0,y_0 \in \{0,1\}$
\begin{equation}\label{eq:ABr}
\begin{aligned}
& A u(x,y)\vert_{y = y_0} = \partial_x\big(a(x,y_0)\partial_x)u(x,y_0) = 0,\\ &B u(x,y)\vert_{x = x_0} = \partial_y\big(b(x_0,y)\partial_y)u(x_0,y) = 0.
\end{aligned}
\end{equation}
Thus, as $0 = Lu\vert_{\partial\Omega} = Au\vert_{\partial\Omega}  + Bu \vert_{\partial\Omega}$ we have by using \eqref{eq:ABr} $Au\vert_{\partial\Omega} = Bu\vert_{\partial\Omega}  = 0$, i.e., $Au,Bu \in \D(A)\cap\D(B)$.
\end{proof}
\begin{lem}\label{lem:dd}
Let $f \in \D(L^2)\subseteq \D(AB)\cap \D(BA)$. Then for every $\lambda > 0$ the solution $u$ to the elliptic problem
\begin{equation}\label{eq:elliA}
\begin{aligned}
& (\lambda I - A) u = f \;\text{ in } \Omega,\\
& u(0,y) = u(1,y) = 0 \;\text{ for almost every } y \in (0,1)
\end{aligned}
\end{equation}
lies in $\D(AB)\cap\D(BA)$.
\end{lem}
\begin{proof}
For $u$ solving \eqref{eq:elliA} and $f\in\D(L^2)\subset \D(B)$ we have $u \in \D(B)$, see \cite{OstS12}. Thus, $Au = \lambda u - f\in \D(B)$, i.e., $ u\in \D(BA)$. Next we show that $u\in \D(AB)$. For this, we compute
\begin{equation}
B u = \partial_y (b \partial_y u ) = \partial_y b \partial_y \big((\lambda I-A)^{-1} f\big) + b \partial_{yy} \big((\lambda I-A)^{-1} f\big).
\end{equation}
The relation $\partial_y (\lambda I-A)^{-1}  = (\lambda I-A)^{-1} \partial_x\big((\partial_ya)\partial_x\big) (\lambda I-A)^{-1}$  (see \cite[Section 7]{OstS12})  shows that
$$
\partial_y (\lambda I-A)^{-1},\ \partial_{yy} (\lambda I-A)^{-1}\,:\, L^2(\Omega)\rightarrow \D(A)\subset L^2(\Omega)
$$
for smooth coefficients $a$. Consequently $Bu \in \D(A)$, which proves the Lemma.
\end{proof}
The above Lemma in particular implies that for all $\lambda>0$
$$
(\lambda I-A)^{-1}\big\vert_{\D(L^2)},\;(\lambda I-B)^{-1}\big\vert_{\D(L^2)}\,:\, \D(L^2)\subset L^2(\Omega) \rightarrow \D(AB)\cap \D(BA)\subset L^2(\Omega).
$$
Thus, by using Cauchy's integral formula we can conclude that
$$
\mathrm{e}^{tA}\big\vert_{\D(L^2)},\;\mathrm{e}^{tB}\big\vert_{\D(L^2)}\,:\, \D(L^2)\subset L^2(\Omega) \rightarrow \D(AB)\cap \D(BA)\subset L^2(\Omega)
$$
 uniformly in $t > 0$.
Hence, by the closed graph theorem  there exists a constant $C$ such that for any $h > 0$,  $\sup_{s,\tau\in [0,h]} \Vert [A,B] \mathrm{e}^{sA}\mathrm{e}^{\tau B} L^{-2}\Vert \leq C$. Furthermore, as for $0\leq \gamma < \frac54$, $\D(L^\gamma)$ is an invariant subspace under the semigroups generated by the split operators, we have for $0\leq\varepsilon < \frac{1}{4}$ that $\sup_{s,\tau\in [0,h]} \Vert L^{-1+\varepsilon } [A,B] \mathrm{e}^{sA}\mathrm{e}^{\tau B} L^{-1-\varepsilon }\Vert \leq C$  uniformly in $h > 0$ (for some constant depending on $\varepsilon$).

Next we study the regularity of the solution to our model problem~\eqref{eq:abM}. Its exact solution is given by the
variation-of-constants formula
\begin{equation*}
u(t) = \mathrm{e}^{tL}u_0 + \int_0^t \mathrm{e}^{(t-\tau)L}
g(\tau)\,\mathrm{d}\tau, \qquad 0 \leq t \leq T.
\end{equation*}
At time $t_{n+1} = t_n +h$, with a step size $h>0$, the solution can be rewritten as
\begin{equation*}
\begin{aligned}
u(t_{n+1}) = \mathrm{e}^{hL}u(t_n) + \int_0^h \mathrm{e}^{(h-s)L} g(t_n+s)\,\mathrm{d}s.
\end{aligned}
\end{equation*}
Taylor expansion of $g(t_n+s)$ at $t_n$ yields
 \begin{equation*}
u(t_{n+1}) =\mathrm{e}^{hL}u(t_n) + \int_0^h \mathrm{e}^{(h-s)L}\Big
(g(t_n) +  sg'(t_n) + \int_{t_n}^{t_n+s} (t_n+s - \tau)
g^{\prime\prime}(\tau)\,\mathrm{d}\tau\Big)\mathrm{d}s.
\label{exakt}
\end{equation*}
We introduce the following notation  (see \cite{HanO09}):
\begin{defn}\label{p1def1}
For $h>0$ we define the bounded operators $\lambda_j$ by setting
\begin{equation*}
\lambda_j = \frac{1}{h^j} \int_0^h \mathrm{e}^{(h-s)L}
\frac{s^{j-1}}{(j-1)!}\,\mathrm{d}s,\qquad j\geq 1
\end{equation*}
and $\lambda_0 = \mathrm{e}^{hL}$. We likewise define $\alpha_0 = \mathrm{e}^{hA},\, \beta_0 = \mathrm{e}^{hB}$ and for $j\geq 1$
\begin{equation*}
\alpha_j = \frac{1}{h^j} \int_0^h \mathrm{e}^{(h-s)A}
\frac{s^{j-1}}{(j-1)!}\,\mathrm{d}s,\qquad \beta_j = \frac{1}{h^j} \int_0^h \mathrm{e}^{(h-s)B}
\frac{s^{j-1}}{(j-1)!}\,\mathrm{d}s.
\end{equation*}
\end{defn}
 With   this notation the exact solution to \eqref{eq:abM} can be written as
\begin{equation}\label{eq:exactR3}
u(t_{n+1}) = \lambda_0 u(t_n) + h\lambda_1 g(t_n) + h^2 \lambda_2
g^{\prime}(t_n) +R_{3,n}(h),
\end{equation}
with the remainder
\begin{equation}\label{eq:R3}
R_{3,n}(h) =\int_0^h \mathrm{e}^{(h-s)L}\Big ( \int_{t_n}^{t_n+s}
(t_n+s - \tau)
g^{\prime\prime}(\tau)\,\mathrm{d}\tau\Big)\,\mathrm{d}s.
\end{equation}
We note for later use that following relations hold on $L^2(\Omega)$
\begin{equation}\label{p1rek}
\lambda_j = \frac{1}{j!} I + hL \lambda_{j+1}, \qquad j \geq 0,
\end{equation}
which can easily be seen by integration by parts.

We now state some well known regularity results for analytic semigroups, see \cite{Lun1995} for further details. This allows us to play back the regularity assumptions on the exact solution to regularity assumptions on the initial data. For some $0\leq \gamma\leq 1$ and $\theta >0$ let the data of problem \eqref{eq:abM} satisfy
\begin{equation}\label{p1eq:reg-data2}
u_0 \in \D(L^\gamma), \qquad Lu_0 + g(0) \in \D(L^\gamma),
\qquad g \in \mathcal{C}^{1+\theta} ([0,T]; L^2(\Omega)),
\end{equation}
then the solution of \eqref{eq:abM} possesses the regularity
\begin{equation}\label{p1eq:regularity}
u\in\mathcal{C}^2([0,T];L^2(\Omega))\cap \mathcal{C}^1([0,T];\D(L^\gamma)).
\end{equation}
Note that if we consider second-order methods, we need more regularity of the solution. For analytic semigroups, this requirement can again be easily expressed in terms of the data. If, for some $0\leq \gamma\leq 1$ and some $\theta > 0$, the data satisfies
\begin{equation}
\begin{aligned}\label{eq:reg-data3}
&g\in \mathcal{C}^{2+\theta}([0,T];L^2(\Omega)),\\
&u_0 \in \D(L^\gamma),\quad Lu_0 + g(0)\in \D(L^\gamma),\quad L^2u_0 + Lg(0) + g'(0)\in \D(L^\gamma),
\end{aligned}
\end{equation}
then the exact solution of the evolution
equation \eqref{eq:abM} satisfies
\begin{equation}\label{eq:high-reg}
u \in \mathcal{C}^{3}([0,T];L^2(\Omega))\cap \mathcal{C}^{2}([0,T];\D(L^\gamma)).
\end{equation}

\section{Convergence results}\label{sec:conv}

In order to state our sharp convergence results we need the following lemma about the smoothing property of splitting methods recently derived in \cite{OstS12}.
\begin{lem}\label{lem:smoothDirLie}
For $0\leq \alpha <1$, the following smoothing properties hold
\begin{equation*}
\begin{aligned}
&\Vert (-L)^{\alpha}\big(\mathrm{e}^{hA}\mathrm{e}^{hB}\big)^n \Vert \leq C t_n^{-\alpha},\qquad 0<t_n=nh\le T,\\
&\Vert (-L)^{\alpha}\big(\mathrm{e}^{\frac{h}{2}A}\mathrm{e}^{hB}\mathrm{e}^{\frac{h}{2}A}\big)^n \Vert \leq C t_n^{-\alpha},\qquad 0<t_n=nh\le T,
\end{aligned}
\end{equation*}
where the constant C can be chosen uniformly on $[0,T]$ and, in particular,
independently of $n\ge 1$ and $h>0$.
\end{lem}
Now we are able to state the main convergence results for the exponential Lie and Strang splitting, see Sections \ref{sec:convLie} and \ref{sec:convStrang}, respectively.

\subsection{The exponential Lie splitting}\label{sec:convLie}

As a numerical scheme, we now consider the exponential Lie splitting method for inhomogeneous equations, i.e., the numerical approximation to the exact solution at time $t_{n+1} = t_n+h$ is given by
\begin{equation}\label{eq:inLie}
u_{n+1} = \mathrm{e}^{hA}\mathrm{e}^{hB}\big(u_n + hg(t_n)\big),
\end{equation}
 and initial condition $u_0 = u(0)$.
Under natural regularity assumptions in time and space on the inhomogeneity  $g$ we obtain first-order convergence.
\begin{thm}\label{thm:inhLie}
The global error of the exponential Lie splitting method \eqref{eq:inLie} satisfies the following bound:   For all fixed $T$ and $\gamma > 0$,   there exists a constant $C$ such that for all $h > 0$, we have
\begin{equation}\label{eq:errglLie}
\Vert u(t_{n})-u_n\Vert \leq C h \sup_{t\in[0,T]}\Big( \Vert L^\gamma u'(t)\Vert + \Vert L^\gamma g(t)\Vert  + \Vert g'(t)\Vert\Big) 
\end{equation}
for all $n$ such that $t_n \leq T$, where $u(t)$ is the exact solution of  \eqref{eq:abM}.
\end{thm}
We recall that the domain $\D(L^\gamma)$ is free of boundary conditions for $\gamma < \frac14$. Thus full-order convergence of the exponential Lie splitting only requires some additional smoothness in space of the inhomogeneity $g$.

Also note that if for some $0\leq \gamma\leq 1$ and $\theta>0$, $u_0 \in \D(L^{1+\gamma})$, $g(0)\in \D(L^\gamma)$ and $g\in \mathcal{C}^{1+\theta}([0,T];L^2(\Omega))$, then the exact solution of problem \eqref{eq:abM} will satisfy the regularity assumption \eqref{p1eq:regularity}, in particular $\sup_{t\in[0,T]} \Vert L^\gamma u'(t)\Vert \leq C$.

\begin{proof}
In the following we will denote by $S$ the exponential Lie splitting, i.e., $S = \mathrm{e}^{hA}\mathrm{e}^{hB}$.
The exact solution of \eqref{eq:abM} is given by
\begin{equation}\label{exli}
u(t_{n+1}) = \lambda_0 u(t_n) + h\lambda_1 g(t_n) + R_{2,n}(h),
\end{equation}
where we have set
\begin{equation}\label{p1eq:R2}
R_{2,n}(h) = \int_0^h \mathrm{e}^{(h-s)L} \Big ( \int_{t_n}^{t_n+s}
g^{\prime}(\tau)\,\mathrm{d}\tau\Big)\,\mathrm{d}s.
\end{equation}
In order to prove convergence of the method, we first take a look at the error $e_{n+1} = u_{n+1}- u(t_{n+1})$. Subtracting \eqref{exli} from \eqref{eq:inLie} leads to
\begin{equation}
\begin{aligned}
e_{n+1} = Se_n + (S-\lambda_0)u(t_n) + h(S-\lambda_1) g(t_n)-R_{2,n}(h).
\end{aligned}
\end{equation}
Using the recurrence relation \eqref{p1rek}, namely $\lambda_1 = \sfrac{1}{h}(\lambda_0 - I)$, as well as the equation in the form $u(t_n)+L^{-1}g(t_n) = L^{-1}u'(t_n)$, we obtain
\begin{equation}\label{eq:ei}
\begin{aligned}
e_{n+1} = Se_n + S\left(u(t_n)+hg(t_n)\right)-\lambda_0 L^{-1}g(t_n)+L^{-1}g(t_n)-R_{2,n}(h).
\end{aligned}
\end{equation}
The convergence analysis of the exponential Lie splitting for the homogeneous problem stated in \cite[proof of Thm.\;3.4]{HanO09} allows us to express the identity operator $I$ restricted to the domain $\D(L)$ as
\begin{equation}\label{eq:idLie}
I\big\vert_{\D(L)}= S(I-hL) + h^2\mathcal{K}_2,
\end{equation}
where  the operator $\mathcal{K}_2$ is such that $h\Vert \mathcal{K}_2L^{-1}\Vert \leq C$ for some constant independent of $h$.   Inserting the representation \eqref{eq:idLie} of the identity operator in front of the terms with $L^{-1}$ in \eqref{eq:ei} yields
\begin{equation*}
\begin{aligned}
e_{n+1} =& Se_n + S\left(u(t_n)+hg(t_n)\right)-I \lambda_0 L^{-1}g(t_n)+I L^{-1}g(t_n)-R_{2,n}(h)\\
=& S e_n + S (I-\lambda_0 + hL\lambda_0) L^{-1}u'(t_n) - h^2\mathcal{K}_2\lambda_0 L^{-1}u'(t_n) + h^2\mathcal{K}_2L^{-1}g(t_n)-R_{2,n}(h).
\end{aligned}
\end{equation*}
Using once more the recurrence relation \eqref{p1rek} we obtain
\begin{equation}\label{eq:localErrLie}
e_{n+1} = Se_n + h^2 S(\lambda_1-\lambda_2)Lu'(t_n)- h^2\mathcal{K}_2\lambda_0 L^{-1}u'(t_n) + h^2\mathcal{K}_2L^{-1}g(t_n)-R_{2,n}(h).
\end{equation}
Solving the above error recursion we obtain with the aid of Lemma \ref{lem:smoothDirLie} for $0<\gamma<1$
\begin{equation}\label{err:lie}
\begin{aligned}
\Vert e_{n+1} &\Vert \leq h^2\sum_{k=0}^{n-1}\Vert S^{n-k} L^{1-\gamma}\Vert \Vert L^{-1+\gamma}J_k\Vert + h^2\Vert J_n\Vert + \sum_{k=0}^{n-k}\Vert S^{n-k}\Vert \Vert R_{2,k}(h)\Vert\\
&\leq Ch\sum_{k=0}^{n-1}\frac{h}{((n-k)h)^{1-\gamma}} \Vert L^{-1+\gamma}J_k\Vert + h^2\Vert J_n\Vert + \sum_{k=0}^{n-k} \Vert R_{2,k}(h)\Vert,
\end{aligned}
\end{equation}
where we have set
\begin{equation*}
J_k = S(\lambda_1-\lambda_2)Lu'(t_k)- \mathcal{K}_2\lambda_0 L^{-1}u'(t_k) + \mathcal{K}_2L^{-1}g(t_k).
\end{equation*}
Using the domain invariance of $\D(L^{1-\gamma})$ for $0<\gamma<1$ under the semigroups, we have
\begin{equation*}
\Vert L^{-1+\gamma} J_k\Vert \leq C \Big( \Vert L^\gamma u'(t_k)\Vert + \Vert L^\gamma g(t_k)\Vert\Big).
\end{equation*}
Note that $h\Vert \mathcal{K}_2L^{-1}\Vert$ is bounded and thus $h^2\Vert J_n\Vert \leq Ch$. Furthermore    $\Vert R_{2,k}(h)\Vert\leq Ch^2\mathrm{sup}_{t\in[0,T]}\Vert g'(t)\Vert$ ,  see \eqref{p1eq:R2}.   Thus, by estimating the sum in \eqref{err:lie} by its Riemann integral, we obtain the desired convergence result \eqref{eq:errglLie}.
\end{proof}

\subsection{The exponential Strang splitting}\label{sec:convStrang}

In this section we are interested in applying a second-order exponential splitting method to the inhomogeneous problem \eqref{eq:abM}. As numerical approximation $u_{n+1}$ to the exact solution $u(t_{n+1})$ we choose the following scheme
\begin{equation}\label{eq:inStrang}
u_{n+1} = \mathrm{e}^{\frac{h}{2}A}\mathrm{e}^{\frac{h}{2}B}\left(\mathrm{e}^{\frac{h}{2}B}\mathrm{e}^{\frac{h}{2}A}u_n + h g\big(t_n+\sfrac{h}{2}\big)\right).
\end{equation}
Clearly there are many possible choices of applying second-order exponential splitting methods to inhomogeneous problems, such as for instant the scheme
\begin{equation}\label{eq:inStrangB}
u_{n+1}  = \mathrm{e}^{\frac{h}{2}A}\mathrm{e}^{hB}\mathrm{e}^{\frac{h}{2}A}\left( u_n + \sfrac{h}{2} g(t_n)\right) + \sfrac{h}{2}g(t_n + h),
\end{equation}
which we will call Strang B splitting in the following. The main motivation behind our choice \eqref{eq:inStrang} lies in the crucial  fact that for $u_0 \in \D(L)$ the numerical solution $u_{n+1}$ remains in the domain $\D(L)$ for every $n\geq 0$, which can be seen from the estimate
\begin{equation*}
\begin{aligned}
\Vert L u_{n+1} \Vert &\leq\big \Vert L\mathrm{e}^{\frac{h}{2}A}\mathrm{e}^{hB}\mathrm{e}^{\frac{h}{2}A} L^{-1}\big\Vert \Vert L u_{n}\Vert + \big\Vert h L \mathrm{e}^{\frac{h}{2}A}\mathrm{e}^{\frac{h}{2}B}\big\Vert \Vert g\big(t_n+\sfrac{h}{2}\big)\Vert\\&
\leq C \Big(\Vert L u_n\Vert +  \Vert g\big(t_n+\sfrac{h}{2}\big)\Vert\Big)
\end{aligned}
\end{equation*}
that uses the domain invariance of $\D(L)$ under the semigroups. In contrast, the numerical solution of Strang B \eqref{eq:inStrangB} remains in the domain $\D(L)$ only under the unnatural assumption that $g(t)\in \D(L)$ for all $t\geq 0$. The favorable choice of \eqref{eq:inStrang} is also confirmed in the numerical experiments given in Section \ref{num}.

\begin{thm}\label{thm:inhStrang}
The global error of the exponential Strang splitting method \eqref{eq:inStrang}  satisfies the following bound:
  For all fixed $T$, $\gamma > 0$,   and $\nu$ with $0 \leq \nu \leq 1$, there exists a constant $C$ such that for all $h > 0$, we have
\begin{equation}\label{thm:glerrSt}
\Vert u(t_n) - u_n\Vert \leq C h^{1+\nu}\sup_{t\in [0,T]}\Big( \Vert L^\gamma u''(t)\Vert + \Vert L^{\gamma + \nu} g(t)\Vert + \Vert L^{\gamma} g'(t)\Vert   + \Vert g''(t)\Vert\Big)
\end{equation}
for all $n$ such that $t_n \leq T$, where $u(t)$ is the exact solution of  \eqref{eq:abM}.
\end{thm}

\begin{rem}\label{rem:ConvStrang}
Note that full-order convergence of the exponential Strang splitting \eqref{eq:inStrang} requires the additional assumption that $g\in\mathcal{C}([0,T];\D(L^{1+\gamma}))$, for some $\gamma >0$. In particular the inhomogeneity $\psi(t,\cdot,\cdot) = g(t)$ needs to vanish at the boundary, i.e., $\psi(t,\cdot,\cdot)\vert_{\partial\Omega} = 0$.

Also note that if for some $0\leq \gamma\leq 1$ and $\theta>0$, $u_0 \in \D(L^{2+\gamma})$, $g(0)\in \D(L^{1+\gamma})$, $g'(0)\in \D(L^\gamma)$ and $g\in \mathcal{C}^{2+\theta}([0,T];L^2(\Omega))$, then the exact solution of problem \eqref{eq:abM} will satisfy the regularity assumption \eqref{eq:high-reg}, in particular $\sup_{t\in[0,T]} \Vert L^\gamma u''(t)\Vert \leq C$.
\end{rem}

\begin{rem}\label{rem:fullorderSt}
Full-order convergence of the exponential Strang splitting \eqref{eq:inStrang} without unnatural assumptions on the inhomogeneity $g(t)$  can be   obtained when considering the evolution equation \eqref{eq:abM} in the extrapolation space $(\D(L)^\ast,\Vert L^{-1}\cdot\Vert)$. We refer to \cite{Ama1995} for an introduction to scales of Banach spaces. Precisely,  we can adapt the previous proof and show that  the global error measured in the $\Vert L^{-1}\cdot\Vert$ norm satisfies the following bound   for any $\gamma > 0$:
\begin{equation*}
\Vert L^{-1}( u(t_n) - u_n)\Vert \leq C h^2\sup_{t\in [0,T]}\Big( \Vert L^{\gamma} u'(t)\Vert + \Vert L^{\gamma} g(t)\Vert + \Vert L^{\gamma-1} g'(t)\Vert   + \Vert L^{-1} g''(t)\Vert \Big)
\end{equation*}
with a constant $C$ that can be chosen uniformly  with respect to $h$  on $[0,T]$.
\end{rem}

\begin{proof}
Let $ S = \mathrm{e}^{\frac{h}{2}A}\mathrm{e}^{hB}\mathrm{e}^{\frac{h}{2}A}$. We start off with the analysis of the local error $e_{n+1} = u_{n+1}-u(t_{n+1})$. Subtracting the exact solution \eqref{eq:exactR3} from the numerical approximation \eqref{eq:inStrang} we obtain
\begin{equation}
\begin{aligned}
e_{n+1} &= Se_n + \big(S-\lambda_0\big)u(t_n) + h\big(\mathrm{e}^{\frac{h}{2}A}\mathrm{e}^{\frac{h}{2}B} -\lambda_1\big)g(t_n) \\&\quad + h^2\big(\sfrac{1}{2}\mathrm{e}^{\frac{h}{2}A}\mathrm{e}^{\frac{h}{2}B} -\lambda_2\big) g'(t_n) -R_{3,n}(h)
\end{aligned}
\end{equation}
with $R_{3,n}$ given in \eqref{eq:R3}. In the following we set
\begin{equation*}
E_{n}(h) = \big(S-\lambda_0\big)u(t_n) + h\big(\mathrm{e}^{\frac{h}{2}A}\mathrm{e}^{\frac{h}{2}B} -\lambda_1\big)g(t_n) + h^2\big(\sfrac{1}{2}\mathrm{e}^{\frac{h}{2}A}\mathrm{e}^{\frac{h}{2}B} -\lambda_2\big) g'(t_n).
\end{equation*}
With the identities $\lambda_1 = \sfrac{1}{h}(\lambda_0 -I)L^{-1}$ and $\lambda_2 = \sfrac{1}{h^2}(\lambda_0 -I - hL)L^{-2}$ we obtain
\begin{equation}\label{eq:ob}
\begin{aligned}
E_{n}(h) &= S u(t_n) -\lambda_0 L^{-2}u''(t_n) + h\mathrm{e}^{\frac{h}{2}A}\mathrm{e}^{\frac{h}{2}B}g(t_n) + L^{-1}g(t_n) \\&\quad + \sfrac{h^2}{2}\mathrm{e}^{\frac{h}{2}A}\mathrm{e}^{\frac{h}{2}B}g'(t_n) + L^{-2}g'(t_n) + hL^{-1}g'(t_n).
\end{aligned}
\end{equation}
Note that the following representations of the identity operator are valid on $\D(L)$ and $\D(L^2)$, respectively, see \cite{HanO09},
\begin{equation}\label{id:str}
\begin{aligned}
& I \vert_{\D(L)}= S(I-hL)+\sfrac{h^2}{2}\mathcal{K}_2,\\
& I \vert_{\D(L^2)}= S(I-hL+\sfrac{h^2}{2}L^2)+h^3\mathcal{K}_3,
\end{aligned}
\end{equation}
where  the operators $\mathcal{K}_2$ and $\mathcal{K}_3$ satisfy the following bounds: for all $\nu$ with $0 \leq \nu \leq 1$
\begin{equation}\label{b1}
h^2 \Vert \mathcal{K}_2 L^{-1}f\Vert \leq C h^{1+\nu}\Vert L^\nu f\Vert,
\end{equation}
and for all $\gamma > 0$ and $\nu$ with $0 \leq \nu \leq 1$
\begin{equation}\label{b2}
h^3\Vert L^{-1+\gamma}\mathcal{K}_3L^{-1} f \Vert \leq C  h^{2+\nu}\Vert L^{\gamma+\nu}f\Vert
\end{equation}
for some constants independent of $h$ and for all functions $f$.
Using the first representation of the identity operator given in \eqref{id:str} in front of the terms with $L^{-1}$ and $L^{-2}$ in \eqref{eq:ob} we obtain
\begin{equation}\label{eq:errSt2L}
\begin{aligned}
E_n(h)  =& (S-\lambda_0)L^{-2}u''(t_n) + h\mathrm{e}^{\frac{h}{2}A}\mathrm{e}^{\frac{h}{2}B}(I- \mathrm{e}^{\frac{h}{2}B}\mathrm{e}^{\frac{h}{2}A})g(t_n) \\&+\textstyle \frac{h^2}{2}\mathcal{K}_2 L^{-1}\big(g(t_n)+L^{-1}g'(t_n)+hg'(t_n)\big)+\textstyle \frac{h^2}{2}\mathrm{e}^{\frac{h}{2}A}\mathrm{e}^{\frac{h}{2}B}g'(t_n)-h^2Sg'(t_n).
\end{aligned}
\end{equation}
From the convergence analysis of the exponential Lie splitting with $h$ replaced by $\sfrac{h}{2}$ we know that
\begin{equation}\label{b1b}
\begin{aligned}
&I-\mathrm{e}^{\frac{h}{2}B}\mathrm{e}^{\frac{h}{2}A} = -\sfrac{h}{2}\mathrm{e}^{\frac{h}{2}B}\mathrm{e}^{\frac{h}{2}A}L+\sfrac{h^2}{4}\mathcal{K}_2,
\end{aligned}
\end{equation}
see \eqref{eq:idLie}. Furthermore, by \cite{HanO09} we have  for all functions $f$ and all $\nu$ with $0\leq\nu\leq 1$
\begin{equation}\label{b1bb}
\Vert  (S-\lambda_0)L^{-2} f\Vert \leq C h^{2+\nu} \Vert L^\nu f\Vert
\end{equation}
for some constant $C$ independent of $h$ and $f$.
Thus, using \eqref{b1b} as well as the bounds \eqref{b1} and \eqref{b1bb} in \eqref{eq:errSt2L}, we obtain that  for $0 \leq \nu \leq 1$
\begin{equation}\label{thm:bound1}
\Vert E_{n}(h)\Vert \leq C  h^{1+\nu}\Big( \Vert  u''(t_n)\Vert + \Vert L^{ \nu} g(t_n)\Vert + \Vert  g'(t_n)\Vert \Big).
\end{equation}
Using the second representation of the identity operator given in \eqref{id:str} in front of the terms with $L^{-1}$ and $L^{-2}$ in \eqref{eq:ob} we obtain
\begin{equation}
\begin{aligned}\label{eq:errSt2}
E_{n}(h) &= \big(S -\lambda_0 \big)L^{-2}u''(t_n) +h\mathrm{e}^{\frac{h}{2}A}\mathrm{e}^{\frac{h}{2}B} \Big(I+\mathrm{e}^{\frac{h}{2}B}\mathrm{e}^{\frac{h}{2}A}\big(\sfrac{h}{2}L-I\big)\Big)g(t_n)\\&\quad
+ \sfrac{h^2}{2}\mathrm{e}^{\frac{h}{2}A}\mathrm{e}^{\frac{h}{2}B} \Big(I-\mathrm{e}^{\frac{h}{2}B}\mathrm{e}^{\frac{h}{2}A}\Big) g'(t_n)+h^3 \mathcal{K}_3 L^{-1}g(t_n)+h^3\mathcal{K}_3L^{-2}g'(t_n)\\&\quad+ h^4 \mathcal{K}_3 L^{-1}g'(t_n).
\end{aligned}
\end{equation}
From the convergence analysis of the exponential Lie splitting with $h$ replaced by $\sfrac{h}{2}$ we know that
\begin{equation}\label{eq:St1}
\begin{aligned}
&I+\mathrm{e}^{\frac{h}{2}B}\mathrm{e}^{\frac{h}{2}A}\Bigl(\sfrac{h}{2}L-I\Bigr)=\sfrac{h^2}{4}\mathcal{K}_2,\\
&I-\mathrm{e}^{\frac{h}{2}B}\mathrm{e}^{\frac{h}{2}A} = -\sfrac{h}{2}\mathrm{e}^{\frac{h}{2}B}\mathrm{e}^{\frac{h}{2}A}L+\sfrac{h^2}{4}\mathcal{K}_2,
\end{aligned}
\end{equation}
see again \eqref{eq:idLie}. Thus, inserting the relation \eqref{eq:St1} into \eqref{eq:errSt2} we obtain
\begin{equation}
\begin{aligned}\label{eq:eeerSt}
E_{n}(h) &= (S-\lambda_0)L^{-2}u''(t_n) + \sfrac{h^3}{4}\mathrm{e}^{\frac{h}{2}A}\mathrm{e}^{\frac{h}{2}B}\mathcal{K}_2g(t_n)
- \sfrac{h^3}{2}SLg'(t_n)\\&\quad+\sfrac{h^4}{8}\mathrm{e}^{\frac{h}{2}A}\mathrm{e}^{\frac{h}{2}B}\mathcal{K}_2g'(t_n)
 +h^3 \mathcal{K}_3 L^{-1}g(t_n)+h^3\mathcal{K}_3L^{-2}g'(t_n)\\&\quad+ h^4 \mathcal{K}_3 L^{-1}g'(t_n).
\end{aligned}
\end{equation}
From the domain invariance of $\D(L)$ under $S$ and \cite{HanO09H}, we know that for $0<\gamma<1$
\begin{equation*}
\Vert L^{-1+\gamma} (S-\lambda_0) L^{-2}u''(t_n)\Vert \leq Ch^3\Vert L^\gamma u''(t_n)\Vert.
\end{equation*}
Thus, using the domain invariance of the domain $\D(L^{1-\gamma})$, $0<\gamma<1$ under the semigroups in \eqref{eq:eeerSt}, as well as the bound given in \eqref{b2}, we obtain for $0\leq \nu \leq 1$ the following estimate
\begin{equation}\label{thm:bound2}
\Vert L^{-1+\gamma} E_{n}(h)\Vert \leq C h^{2+\nu}\Big( \Vert L^\gamma u''(t_n)\Vert + \Vert L^{\gamma + \nu} g(t_n)\Vert + \Vert L^{\gamma} g'(t_n)\Vert \Big).
\end{equation}
Thus, using the bounds given in \eqref{thm:bound1} and \eqref{thm:bound2} as well as the smoothing property for the exponential Strang splitting, see Lemma \ref{lem:smoothDirLie}, we obtain for $0\leq \nu \leq 1$
\begin{equation}
\begin{aligned}\label{glerr:Strang}
\Vert e_{n+1} \Vert &\leq \sum_{k=0}^{n-1}\Vert S^{n-k} L^{1-\gamma}\Vert \Vert L^{-1+\gamma}E_k(h)\Vert + \Vert E_n(h)\Vert + \sum_{k=0}^{n-k}\Vert S^{n-k}\Vert \Vert R_{3,k}(h)\Vert\\
&\leq Ch^{1+\nu}\sum_{k=0}^{n-1}\frac{h}{((n-k)h)^{1-\gamma}}  \Big( \Vert L^\gamma u''(t_k)\Vert + \Vert L^{\gamma + \nu} g(t_k)\Vert + \Vert L^{\gamma} g'(t_k)\Vert \Big) \\&\quad+ Ch^{1+\nu}\Big( \Vert  u''(t_n)\Vert + \Vert L^{ \nu} g(t_n)\Vert + \Vert  g'(t_n)\Vert \Big) + \sum_{k=0}^{n-k} \Vert R_{3,k}(h)\Vert.
\end{aligned}
\end{equation}
  By \eqref{eq:R3} we have $\Vert R_{3,k}(h)\Vert \leq C h^3\mathrm{sup}_{t\in[0,T]}\Vert g''(t)\Vert$.    Hence, approximating the sum in \eqref{glerr:Strang} by its Riemann integral, leads to the desired estimate \eqref{thm:glerrSt}.
\end{proof}

\section{Application to inhomogeneous Dirichlet boundary conditions}\label{inhB}

The main aim of this section is to  understand the disappointing behavior of splitting methods when applied to problems involving inhomogeneous boundary conditions with the aid of the convergence results derived in the previous section.  Our model problem in this section thus reads
\begin{equation}\label{eq:abMH}
\begin{aligned}
& \partial_t w(t,x,y) = \mathcal{L}(\partial_x,\partial_y) w(t,x,y),\quad (x,y)\in \Omega=(0,1)^2,\ t\in ]0,T],\\
&w(0,x,y) = w_0(x,y),\\
&w(t,\cdot,\cdot)\big\vert_{\partial\Omega} = f(t,\cdot,\cdot) \quad\text{for all } t\in [0,T],
\end{aligned}
\end{equation}
where $\mathcal{L}(\partial_x,\partial_y) = \partial_x(a(x,y)\partial_x)+\partial_y(b(x,y)\partial_y)$. 
In the following we denote by $Gf(t)$ the Dirichlet map, i.e., the solution to the elliptic problem
\begin{equation*}
\begin{aligned}
&\mathcal{L}(\partial_x,\partial_y) v= 0 \text{ in } \Omega,\\
& v\vert_{\partial\Omega} = f(t).
\end{aligned}
\end{equation*}
In Section \ref{sec:L2} we will discuss the case where the Dirichlet map $Gf(t)$ is analytically known and smooth in space and time, whereas in Section \ref{sec:dual} we focus on the more common situation where $f$ is only known as a boundary function.

\subsection{Formulation as an inhomogeneous equation in $L^2(\Omega)$}\label{sec:L2}
Here we assume that the Dirichlet map $Gf(t)=: F(t)$ is analytically known and in particular smooth in space and time, i.e., our model problem \eqref{eq:abMH} satisfies
\begin{equation}\label{eq:abMH2}
w(t,\cdot,\cdot)\big\vert_{\partial\Omega} = F(t,\cdot,\cdot)\big\vert_{\partial\Omega} \quad\text{for all } t\in [0,T].
\end{equation}

The ansatz is to rewrite \eqref{eq:abMH}, \eqref{eq:abMH2} as an inhomogeneous evolution equation in $L^2(\Omega)$ with homogeneous Dirichlet boundary conditions. Employing the transformation $U(t,x,y) = w(t,x,y) -  F(t,x,y)$, we see that $U$ is the solution to the problem
\begin{equation}\label{eq:homi}
\begin{aligned}
& \partial_t U(t,x,y) = \mathcal{L}(\partial_x,\partial_y) U(t,x,y)+\psi(t,x,y),\quad (x,y)\in \Omega=(0,1)^2,\ t\in ]0,T],\\
&U(0,x,y) = U_0(x,y),\\
&U(t,\cdot,\cdot)\big\vert_{\partial\Omega} =0
\end{aligned}
\end{equation}
with $\psi(t,x,y) = \mathcal{L}(\partial_x,\partial_y)  F(t,x,y)-\partial_t  F(t,x,y)$  and $U_0(x,y) = w_0(x,y) - F(0,x,y)$.  This equation can be formulated as an abstract evolution equation in $L^2(\Omega)$ with $u(t) = U(t,\cdot,\cdot)$ and $g(t) = \psi(t,\cdot,\cdot)$, i.e.,
\begin{equation}\label{eq:abMHD}
\begin{aligned}
&u'(t) = Lu(t)+g(t), \quad t\in]0,T],\\
&u(0) = u_0,
\end{aligned}
\end{equation}
where now the problem is equipped with homogeneous Dirichlet boundary conditions, i.e., $\D(L) = H^2(\Omega)\cap H_0^1(\Omega)$. Thus, we can apply the convergence results for the exponential Lie and Strang splitting stated in Section \ref{sec:conv}.

We are in particular interested in how the regularity assumptions on the initial value $u_0$ and the inhomogeneity $g$ of the reformulated inhomogeneous problem \eqref{eq:abMHD} relate to the regularity assumptions on $w_0$ and the boundary data $F$ of the original inhomogeneous boundary value problem \eqref{eq:abMH}. Therefore we introduce the operator $L_{F(0)} = \mathcal{L}(\partial_x,\partial_y)$ with the domain
\begin{equation*}
\D(L_{F(0)}) = \left\{ \varphi\in H^2(\Omega) \,;\, \varphi\vert_{\partial\Omega} = F(0,\cdot,\cdot)\vert_{\partial\Omega}\right\}.
\end{equation*}

First-order convergence of the exponential Lie splitting requires the following regularity assumption on the initial value $u_0$ and inhomogeneity $g$
\begin{equation}\label{ass:inhLie}
u_0 \in \D(L^\gamma), \qquad u'(0)=Lu_0 + g(0) \in \D(L^\gamma),
\qquad g \in \mathcal{C}^{1+\theta} ([0,T]; L^2(\Omega))
\end{equation}
for some $\gamma >0$, see Theorem \ref{thm:inhLie}. Expressed in terms of the original data in \eqref{eq:abMH}, assumption \eqref{ass:inhLie} reads
\begin{equation*}
\begin{aligned}
&u_0 = w_0- F(0) \in \D(L^\gamma),\\&u'(0) =  w'(0)- F'(0)= L_{F(0)} w(0) - F'(0) \in \D(L^\gamma),\\
& \mathcal{L}(\partial_x,\partial_y) F(t,\cdot,\cdot)-\partial_t  F (t,\cdot,\cdot)\in \mathcal{C}^{1+\theta} ([0,T]; L^2(\Omega)).
\end{aligned}
\end{equation*}
Thus, for first order convergence of the exponential Lie splitting method, the initial value of the original problem \eqref{eq:abMH} needs to satisfy for some $\gamma >0$
\begin{equation}\label{eq:blie}
w_0 \in \D(L_{F(0)})\cap H^{2+2\gamma}(\Omega) = \left\{\varphi\in  H^{2+2\gamma}(\Omega)\,;\, \varphi\vert_{\partial\Omega} = F(0,\cdot,\cdot)\vert_{\partial\Omega}\right\}.
\end{equation}

Hence upon using Theorem \ref{thm:inhStrang},
a convergence rate of order $1.25-\varepsilon$ of the exponential Strang splitting method can be achieved under the additional regularity assumptions
\begin{equation*}
\begin{aligned}
&g\in \mathcal{C}^{2+\theta}([0,T];L^2(\Omega)),\\
&u_0 \in \D(L^\gamma),\quad
u'(0)=Lu_0 + g(0)\in \D(L^\gamma),\\
&u''(0)=L^2u_0 + Lg(0) + g'(0)\in \D(L^\gamma),
\end{aligned}
\end{equation*}
for some $\gamma >0$.   In terms of the solution $w$ and boundary data $F$ of the original problem \eqref{eq:abMH} the additional assumptions read
\begin{equation*}
\begin{aligned}
& u''(0) = w''(0) -  F''(0) = L_{F(0)} ^2w(0)- F''(0)\in \D(L^\gamma),\\
&  \mathcal{L}(\partial_x,\partial_y) F(t,\cdot,\cdot)-\partial_t  F (t,\cdot,\cdot)\in \mathcal{C}^{2+\theta} ([0,T]; L^2(\Omega)).
\end{aligned}
\end{equation*}
Thus, convergence of order $1.25-\varepsilon$ requires the additional smoothness assumption on the initial value $w_0$
\begin{equation}\label{eq:bstrang}
\begin{aligned}
w_0 \in   \,\D(L^2_{F(0)})\cap H^{4+2\gamma}(\Omega) = \left\{\varphi\in  H^{4+2\gamma}(\Omega)\,;\, \mathcal{L}(\partial_x,\partial_y) \varphi \vert_{\partial\Omega} = \varphi \vert_{\partial\Omega} =F(0,\cdot,\cdot)\vert_{\partial\Omega} \right\}
\end{aligned}
\end{equation}
for some $\gamma >0$.

Note that in order to have full-order convergence of the Strang splitting applied to inhomogeneous equations, the inhomogeneity $g$ needs to lie in the domain $\D(L^{1+\gamma})$ for some $\gamma>0$, precisely $g\in\mathcal{C}([0,T]; \D(L^{1+\gamma}))$, see Remark \ref{rem:ConvStrang}. In terms of the original problem \eqref{eq:abMH} this leads to the following assumption on the boundary data $F$
\begin{equation*}
\mathcal{L}(\partial_x,\partial_y) F(t,\cdot,\cdot)-\partial_t  F (t,\cdot,\cdot)\in \mathcal{C} ([0,T]; H^{2+2\gamma}(\Omega)\cap H_0^1(\Omega))
\end{equation*}
for some $\gamma > 0$. We summarize our findings in the following corollary.

\begin{cor}\label{cor:bc0}
The exponential Lie and Strang splitting applied to \eqref{eq:abMH2} are convergent of order $1$, respectively, $1.25-\varepsilon$ (for every $\varepsilon>0$) under the natural assumptions \eqref{eq:blie} and \eqref{eq:bstrang}, respectively. Full second-order convergence of the exponential Strang splitting method is obtained if the boundary data $F$ satisfy
$$
\partial_t  F (t,\cdot,\cdot)\vert_{\partial\Omega} = \mathcal{L}(\partial_x,\partial_y) F(t,\cdot,\cdot)\vert_{\partial\Omega}
$$
for all $t\in [0,T]$.
\end{cor}

\subsection{Formulation as an inhomogeneous equation in the dual space $(\D(L)^\ast, \Vert L^{-1}\cdot\Vert)$}\label{sec:dual}

Here the Ansatz is to rewrite the original problem \eqref{eq:abMH} as an abstract evolution equation in the extrapolation space $\D(L)^\ast$ with the corresponding $\Vert L^{-1}\cdot\Vert$ norm. First we formulate the problem in $L^2(\Omega)$ as follows
\begin{equation}\label{eq:abMHL}
\begin{aligned}
&L^{-1}y'(t) = y(t) + Gf(t),\\
&y(0) = y_0=w(0),
\end{aligned}
\end{equation}
see \cite{Las86} for further details. In order to analyze the convergence of the splitting methods applied to \eqref{eq:abMHL}, we need to reformulate the equation as an evolution equation in the extrapolation space $(\D(L)^\ast, \Vert L^{-1}\cdot\Vert)$, i.e.,
\begin{equation}\label{eq:yLL}
\begin{aligned}
&y'(t) = Ly(t) + LGf(t),\\
&y(0) = y_0.
\end{aligned}
\end{equation}
Note that this is the usual approach if the Dirichlet map of the boundary data is not known analytically.  For instance, in a simple one dimensional situation with constant coefficients, namely
\begin{equation*}
\begin{aligned}
& \partial_t w(t,x) = \partial_{xx} w(t,x),\quad x\in (0,1),\ t\in ]0,T],\\
&w(0,x) = w_0(x),\\
&w(t,0)=\alpha(t),\quad w(t,1) = \beta(t),
\end{aligned}
\end{equation*}
employing standard second-order central differences with grid size $\Delta x$, yields  a semidiscrete problem of type
\begin{equation*}
W'(t) = \frac{1}{\Delta x^2}\text{diag}\,[1,-2,1]W(t) + \frac{1}{\Delta x^2}\big[\alpha(t),0,\ldots,0,\beta(t)\big]^{\sf T},
\end{equation*}
which is of the form \eqref{eq:yLL}.

First, we analyze the convergence in the dual space $(\D(L)^\ast, \Vert L^{-1}\cdot\Vert)$ and consider
\begin{equation}\label{eq:uLL}
\begin{aligned}
& u'(t) = Lu(t) + Gf(t),\\
& u(0 ) = L^{-1} y_0
\end{aligned}
\end{equation}
with $u(t) = L^{-1} y(t)$. Note that due to this transformation, we only obtain a convergence result for the original solution $y(t)$ in the extrapolation space $\D(L)^\ast$, i.e., with respect to the $\Vert L^{-1}\cdot\Vert$ norm.

First-order convergence of the exponential Lie splitting in $\D(L)^\ast$ requires the regularity assumptions
\begin{equation*}
y_0 \in \D(L^{\gamma-1}),\ y'(0)=Ly_0 + LGf(0) \in \D(L^{\gamma-1}), \ f\in \mathcal{C}^{1+\theta}([0,T]; L^2(\partial \Omega))
\end{equation*}
for some $\gamma >0$. Thus, the initial value $y_0$ and boundary data $f$ need to satisfy
\begin{equation}\label{eq:bblie}
y_0 \in H^{2\gamma}(\Omega), \quad f\in \mathcal{C}^{1+\theta}([0,T]; L^2(\partial \Omega))
\end{equation}
for some $\gamma >0$.

Convergence of order $1.25-\varepsilon$ of the exponential Strang splitting in $\D(L)^\ast$ requires
\begin{equation*}
y''(0) =L^2y_0 + L^2Gf(0)+LGf'(0) \in \D(L^{\gamma-1}), \quad f\in \mathcal{C}^{2+\theta}([0,T]; L^2(\partial \Omega)).
\end{equation*}
Thus, the initial value $y_0$ needs to satisfy
\begin{equation}\label{eq:bbstrang}
y_0 \in \D(L_{f(0)})\cap H^{2+2\gamma}(\Omega) = \left\{ \varphi\in H^{2+2\gamma}(\Omega)\,;\, \varphi\vert_{\partial\Omega} = f(0)\right\}
\end{equation}
for some $\gamma >0$.

Considering the convergence in $L^2(\Omega)$, i.e., measuring the error of the splitting methods applied to \eqref{eq:yLL} in the standard $L^2(\Omega)$ norm, both splitting methods suffer from an order reduction to $\frac{1}{4}-\varepsilon$ for all $\varepsilon >0$,   as illustrated in  Figure \ref{figureBC}  below.  This can easily be seen in the local error representation of the splitting methods, replacing the inhomogeneity $g(t_n)$ with $Lg(t_n)$ in \eqref{eq:localErrLie} and \eqref{eq:eeerSt}, respectively. For the Lie splitting and $g$ replaced by $Lg$ in the local error \eqref{eq:localErrLie}, the term $h^2\mathcal{K}_2 g(t_n)$ can only be bounded by the loss of powers in $h$, as $\mathcal{K}_2^\gamma g(t_n)$ is only bounded for $\gamma<\frac{1}{4}-\varepsilon$, i.e.,
\begin{equation*}
h^2\Vert\mathcal{K}_2 g(t_n)\Vert \leq h^{1+\frac{1}{4}-\varepsilon} \Vert \mathcal{K}_2^{\frac{1}{4}-\varepsilon} g(t_n)\Vert.
\end{equation*}
Thus, the global error of the exponential Lie splitting in $L^2(\Omega)$ reduces to $\frac{1}{4}-\varepsilon$. Similarly, one can show the same for the exponential Strang splitting method. A numerical experiment where the inhomogeneity is simply chosen $g(t)=1$ confirms this theoretical observation, see Figure \ref{figureBC}.

Again we summarize our findings in a corollary.
\begin{cor}\label{cor:bc}
Measuring the error with respect to the $\Vert L^{-1}\cdot\Vert$ norm the exponential Lie and Strang splitting applied to \eqref{eq:yLL} are convergent with order $1$, respectively, $1.25-\varepsilon$ (for every $\varepsilon>0$) under the natural assumptions \eqref{eq:bblie} and \eqref{eq:bbstrang}, respectively. With respect to the standard $L^2(\Omega)$ norm both methods suffer from an order reduction to $0.25-\varepsilon$ (for every $\varepsilon>0$).
\end{cor}

\section{Numerical experiments}\label{num}

In this section we numerically analyze the order of convergence of the exponential Lie \eqref{eq:inLie}, Strang \eqref{eq:inStrang} and Strang B \eqref{eq:inStrangB} splitting to underline the theoretical convergence results derived in Section \ref{sec:conv}. The main interest lies in the sharpness of the convergence bounds for the exponential Strang splitting for inhomogeneous equations given in Theorem \ref{thm:inhStrang}, as well as the convergence bound for the exponential Lie and Strang splitting for general inhomogeneous Dirichlet boundary conditions given in Corollary \ref{cor:bc}. In the numerical experiments we use standard symmetric finite differences for the spatial discretization of the split differential operators $\mathcal{A}(\partial_x)=\partial_x\bigl(a(x,y) \partial_x\bigr)$ and $\mathcal{B}(\partial_y)=\partial_y\bigl(b(x,y) \partial_b\bigr)$.

In Examples \ref{ex:red} and \ref{ex:full}  presented below, we solve the inhomogeneous model problem \eqref{eq:model} with different choices of inhomogeneities $\psi$. The numerical results confirm the error bound \eqref{thm:glerrSt}. In Example \ref{ex:inhBC} we solve the inhomogeneous boundary problem \eqref{eq:yLL} with the simplest choice of inhomogeneity $g(t) = 1$. The numerical result confirms the severe order reduction stated in Corollary \ref{cor:bc}.

\begin{ex}[Order reduction of the exponential Strang splitting]\label{ex:red}
In this example we choose the following data in \eqref{eq:model}:
\begin{equation*}
\begin{aligned}
& \mathcal{L}(\partial_x,\partial_y) = \partial_x\big( (2xy+3)\partial_x\big)+\partial_y\big((2xy^4+1) \partial_y\big),\\
& w_0(x,y) = \mathrm{e}^{8-\frac{1}{x(1-x)}-\frac{1}{y(1-y)}},\\ 
& \psi(t,x,y)= x(1-x)y(1-y)+t\mathrm{e}^{x^3y}.
\end{aligned}
\end{equation*}
Note that the regularity assumptions \eqref{p1eq:regularity} and \eqref{eq:high-reg} are satisfied as in particular $\psi(0,\cdot,\cdot)\vert_{\partial\Omega} = 0$. Furthermore note that $\psi(t,\cdot,\cdot)\vert_{\partial\Omega} \neq 0$ for $t>0$. Thus, $\gamma+\nu$ must be chosen smaller than one quarter in \eqref{thm:glerrSt}.

Hence, we expect first-order convergence of the exponential Lie splitting \eqref{eq:inLie} and convergence of order $1.25-\varepsilon$ for the exponential Strang splitting method \eqref{eq:inStrang} applied to \eqref{eq:model} formulated in $L^2(\Omega)$. On the contrary, we expect full-order convergence of both methods when formulating the evolution equation \eqref{eq:model} in the extrapolation space $(\D(L)^\ast,\Vert L^{-1}\cdot\Vert)$, i.e., when the error is measured with respect to the $\Vert L^{-1}\cdot\Vert$ norm, cf.~Remark \ref{rem:fullorderSt}. The numerical experiments go in line with our theoretical results, see Figure \ref{figure1} for the $L^2(\Omega)$ and Figure \ref{figure1ex} for the $\D(L)^\ast$ formulation.

\begin{figure}[t]
\centering
\includegraphics[width=66mm]{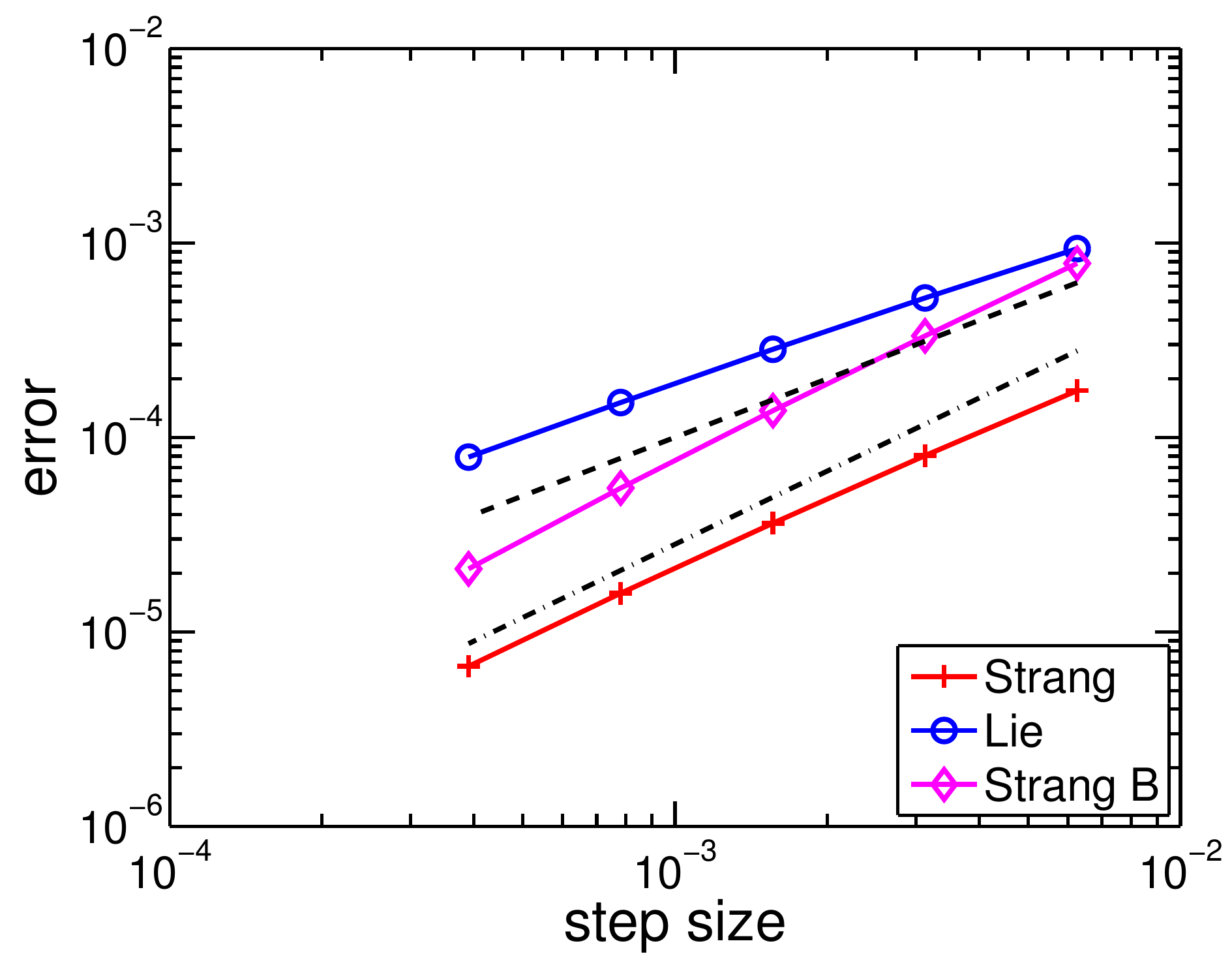}
\caption{Numerical order of the exponential Lie, Strang and Strang B splitting applied to equation \eqref{eq:model} with the data as in Example \ref{ex:red}. The error is measured in a discrete $L^2$ norm. The slopes of the dashed and dash-dotted lines are one and $1.25$, respectively.}\label{figure1}
\end{figure}
\begin{figure}[t]
\centering
\includegraphics[width=66mm]{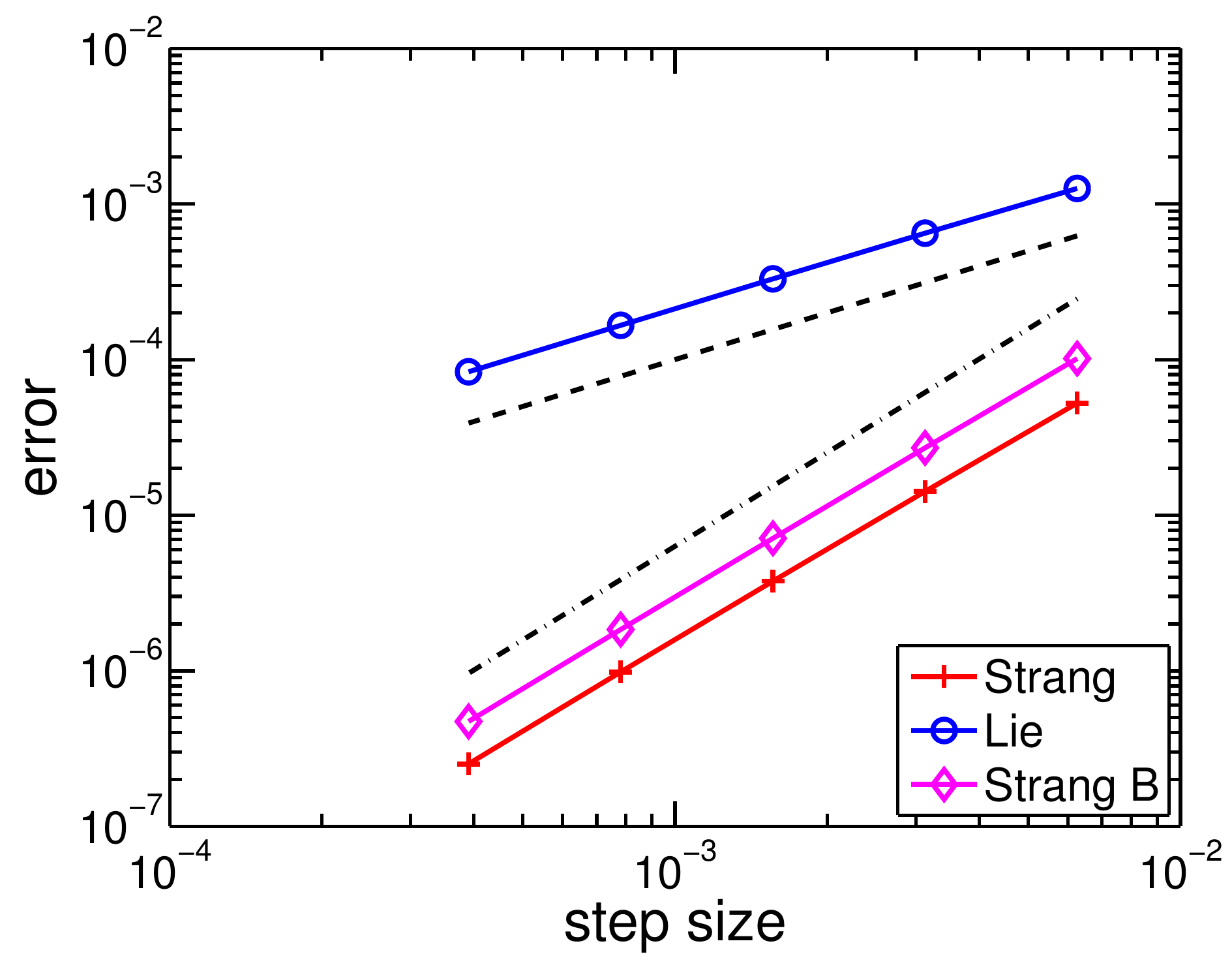}
\caption{Numerical order of the exponential Lie, Strang and Strang B splitting applied to equation \eqref{eq:model} with the data as in Example \ref{ex:red}. The error is measured in a discrete $\D(L)^\ast$ norm, i.e., with respect to $\Vert L^{-1}\cdot\Vert_{L^2}$. The slopes of the dashed and dash-dotted lines are one and two, respectively.}\label{figure1ex}
\end{figure}
\end{ex}

\newpage

\begin{ex}[Full-order convergence of the exponential Strang splitting]\label{ex:full}
Here we choose the following data in \eqref{eq:model}:
\begin{equation*}
\begin{aligned}
& \mathcal{L}(\partial_x,\partial_y) = \partial_x\big( (2xy+3)\partial_x\big)+\partial_y\big((2xy^4+1) \partial_y\big),\\
& w_0(x,y) = \mathrm{e}^{8-\frac{1}{x(1-x)}-\frac{1}{y(1-y)}},\\
& \psi(t,x,y)= x(1-x)y(1-y)\mathrm{e}^{t}.
\end{aligned}
\end{equation*}
Note that the regularity assumptions \eqref{p1eq:regularity} and \eqref{eq:high-reg} are satisfied as $\psi(0,\cdot,\cdot)\vert_{\partial\Omega} = 0$ . In particular $g(t)=\psi(t,\cdot,\cdot)\in \D(L^{1+\gamma})$ for $0\leq \gamma<\sfrac{1}{4}$ and all $t$. Hence, $\nu$ can be chosen equal to one in \eqref{thm:glerrSt}.

Thus, we expect first-order convergence of the exponential Lie splitting \eqref{eq:inLie} and second-order convergence of the exponential Strang splitting method \eqref{eq:inStrang} applied to \eqref{eq:model}. The numerical experiment goes in line with our theoretical results, see Figure \ref{figure2}.
\begin{figure}[t]
\centering
\includegraphics[width=66mm]{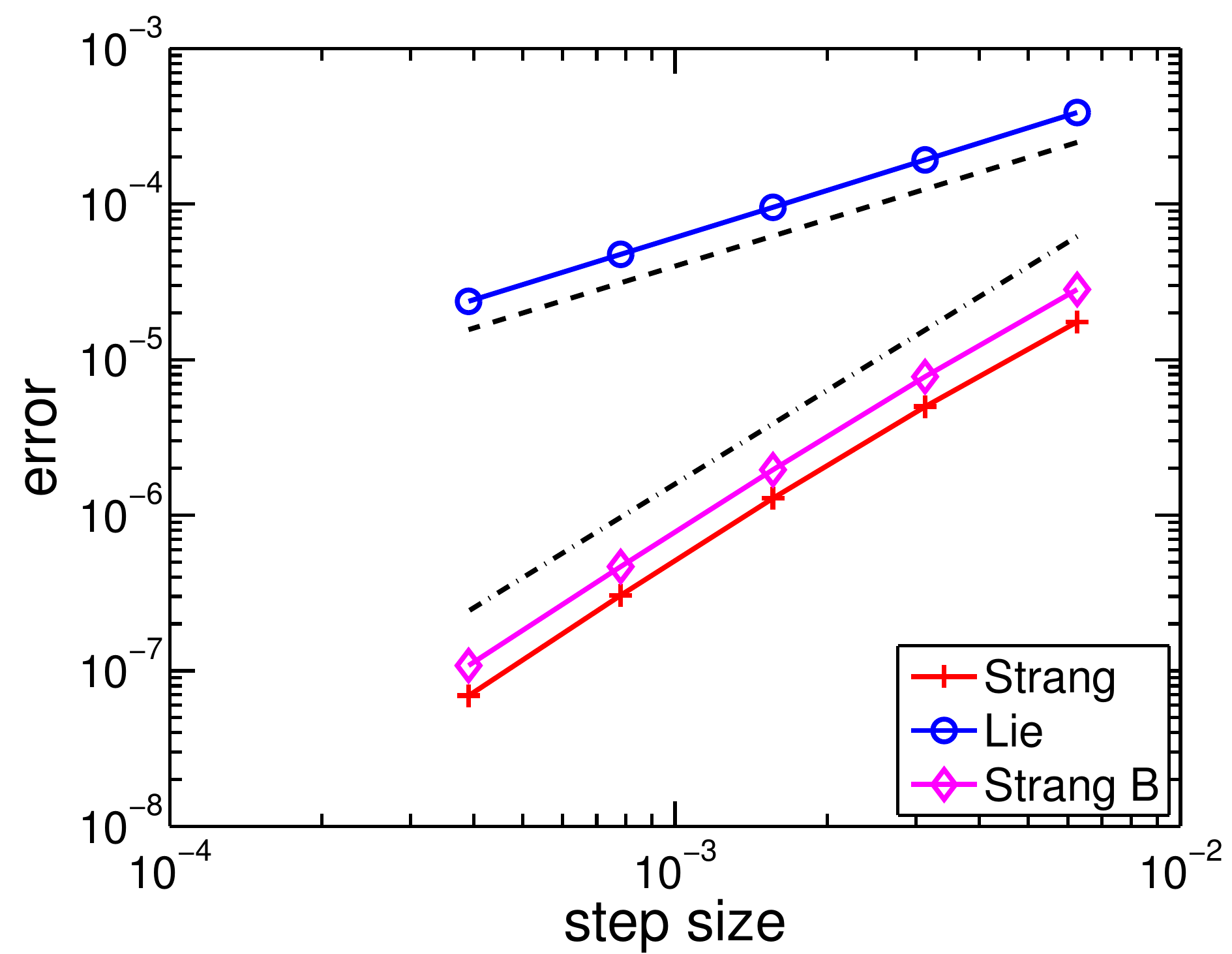}
\caption{Numerical order of the exponential Lie, Strang and Strang B splitting applied to equation \eqref{eq:model} with the data as in Example \ref{ex:full}. The error is measured in a discrete $L^2$ norm. The slopes of the dashed and dash-dotted lines are one and two, respectively.}\label{figure2}
\end{figure}
\end{ex}

\begin{ex}[Inhomogeneous Dirichlet boundary conditions]\label{ex:inhBC}
Here we model the inhomogeneous boundary problem \eqref{eq:yLL} with the data
\begin{equation*}
g(t) = 1,\quad y_0(x,y) = \mathrm{e}^{8-\frac{1}{x(1-x)}-\frac{1}{y(1-y)}}.
\end{equation*}
In Figure \ref{figureBC} we clearly see the severe order reduction of the exponential Lie and Strang splitting in $L^2(\Omega)$, which underlines the sharpness of Corollary \ref{cor:bc}.
\begin{figure}[h!]
\centering
\includegraphics[width=66mm]{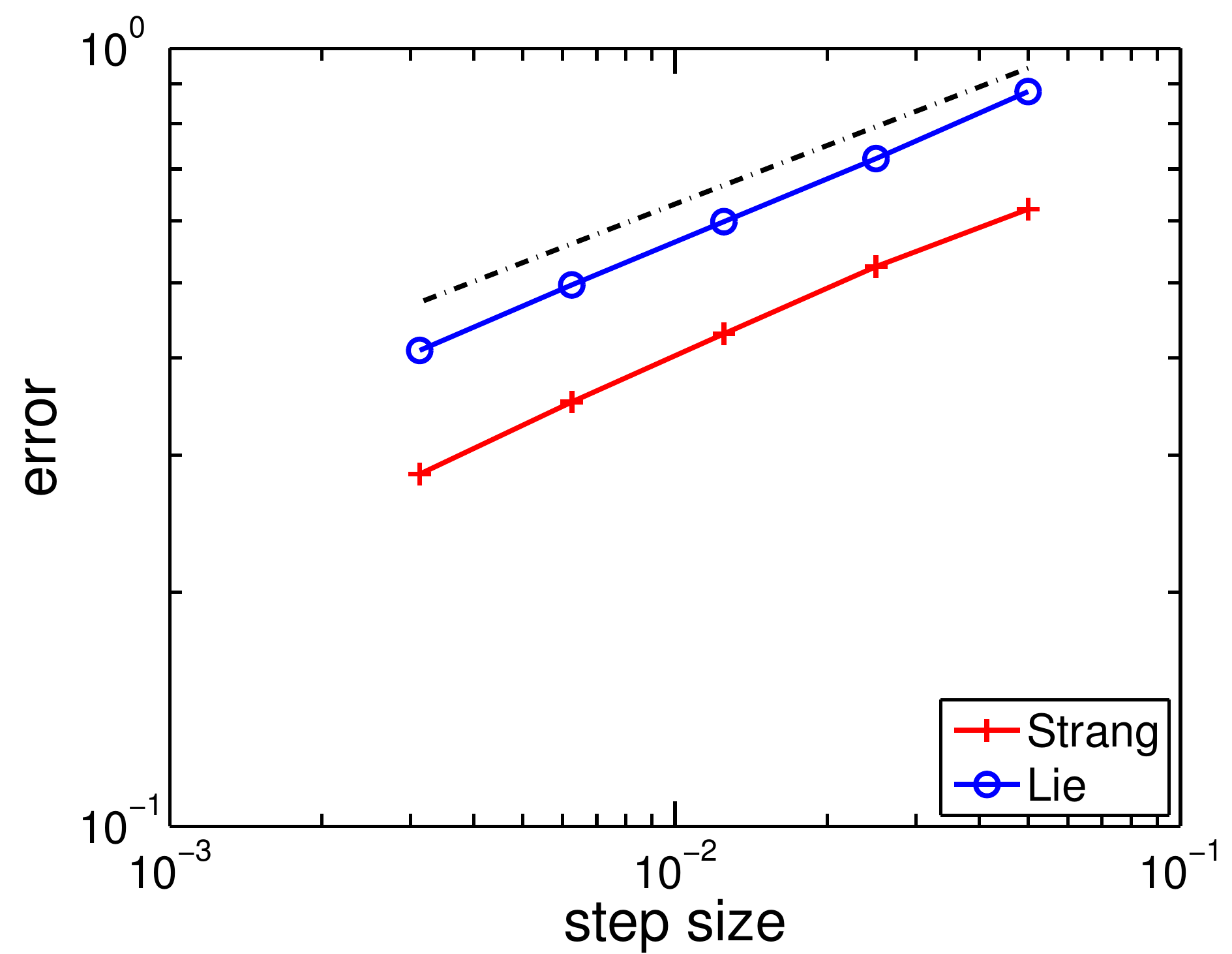}
\caption{Numerical order of the exponential Lie and Strang splitting applied to equation \eqref{eq:yLL} with $g(t) = 1$. The error is measured in a discrete $L^2$ norm. The slope of the dashed line is one quarter.}\label{figureBC}
\end{figure}
\end{ex}
\bibliographystyle{plain}
\bibliography{literatur}

\begin{thebibliography}{10}

\bibitem{Ama1995}
H.~Amann.
\newblock {\em Linear and Quasilinear Parabolic Problems, Volume I}.
\newblock Birkh\"auser, Basel, 1995.

\bibitem{Fujiwara67}
D.~Fujiwara.
\newblock Concrete characterization of the domains of fractional powers of some
  elliptic differential operators of the second order.
\newblock {\em Proc. Japan Acad.}, 43:82--86, 1967.

\bibitem{Grisvard1985}
P.~Grisvard.
\newblock {\em Elliptic Problems in Nonsmooth Domains}.
\newblock Pitman, 1985.

\bibitem{HanO09}
E.~Hansen and A.~Ostermann.
\newblock Exponential splitting for unbounded operators.
\newblock {\em Math. Comp}, 78:1485--1496, 2009.

\bibitem{HanO09H}
E.~Hansen and A.~Ostermann.
\newblock High order splitting methods for analytic semigroups exist.
\newblock {\em BIT}, 49:527--542, 2009.

\bibitem{HeSW}
T.~Hell.
\newblock Compatibility conditions for the {D}irichlet problem of a strictly
  elliptic operator on a rectangle.
\newblock {\em In preparation}.

\bibitem{HunV2003}
W.~Hundsdorfer and J.G. Verwer.
\newblock {\em Numerical Solution of Time-Dependent
  Advection-Diffusion-Reaction Equations}.
\newblock Springer, Berlin Heidelberg, 2003.

\bibitem{JahL00}
T.~Jahnke and C.~Lubich.
\newblock Error bounds for exponential operator splittings.
\newblock {\em BIT}, 40:735--744, 2000.

\bibitem{Las86}
I.~Lasiecka.
\newblock Galerkin approximations of abstract parabolic boundary value problems
  with rough boundary data -- ${L}_p$ theory.
\newblock {\em Math. Comp.}, 47:77--101, 1986.

\bibitem{Lun1995}
A.~Lunardi.
\newblock {\em Analytic Semigroups and Optimal Regularity in Parabolic
  Problems}.
\newblock Birkh\"auser, Basel, 1995.

\bibitem{McLacQ02}
R.I. McLachlan and G.R.W. Quispel.
\newblock Splitting methods.
\newblock {\em Acta Numer.}, 11:341--434, 2002.

\bibitem{OstS12}
A.~Ostermann and K.~Schratz.
\newblock Stability of exponential operator splitting methods for
  non-contractive semigroups.
\newblock {\em To appear in SIAM J. Numer. Anal.}

\bibitem{OstS12b}
A.~Ostermann and K.~Schratz.
\newblock Error analysis of splitting methods for inhomogeneous evolution
  equations.
\newblock {\em Appl. Numer. Math.}, 62:1436--1446, 2012.

\bibitem{Pazy1983}
A.~Pazy.
\newblock {\em Semigroups of Linear Operators and Applications to Partial
  Differential Equations}.
\newblock Springer, New York, 1983.

\end{thebibliography}

\end{document}